\documentclass[a4paper,10pt]{amsart}
\usepackage{CJK}
\usepackage{amsmath}
\usepackage{amsthm,amssymb,latexsym}
\usepackage{amsmath,calligra}
\usepackage{url,ae}
\usepackage{t1enc}
\usepackage{eufrak}
\usepackage{amsfonts}
\usepackage{graphicx}
\usepackage{hyperref}
\usepackage{amscd}
\usepackage{mathrsfs}
\usepackage{bbm}
\usepackage{color}
\usepackage{txfonts}

\usepackage{multirow,makecell}
\usepackage[format=hang,font=small,textfont=it]{caption}

\newtheorem{theorem}{Theorem}[section]
\newtheorem{corollary}[theorem]{Corollary}
\newtheorem{lemma}[theorem]{Lemma}
\newtheorem{proposition}[theorem]{Proposition}

\newtheorem{question}{Question}[section]

\theoremstyle{definition}
\newtheorem{definition}[theorem]{Definition}
\theoremstyle{remark}
\newtheorem{remark}[theorem]{Remark}
\theoremstyle{example}

\newcommand{\be}{\begin{equation}}
\newcommand{\ee}{\end{equation}}
\newcommand{\bea}{\begin{eqnarray}}
\newcommand{\eea}{\end{eqnarray}}
\newcommand{\ben}{\begin{eqnarray*}}
\newcommand{\een}{\end{eqnarray*}}
\newcommand{\bet}{\begin{equation}
\begin{split}}
\newcommand{\eet}{\end{split}
\end{equation}}

\DeclareMathOperator{\Ext}{\mathscr{E}\text{\kern -3pt {\calligra\Large xt}}\,\,}

\begin{document}

\title[Multiplier ideal sheaves on complex spaces with singularities]
      {A note on multiplier ideal sheaves on complex spaces with singularities}

\author{Zhenqian Li}

\date{\today}
\thanks{E-mail: lizhenqian@amss.ac.cn}

\begin{abstract}
The goal of this note is to present some recent results of our research concerning multiplier ideal sheaves on complex spaces and singularities of plurisubharmonic functions. We firstly introduce multiplier ideal sheaves on complex spaces (\emph{not} necessarily normal) via Ohsawa's extension measure, as a special case of which, it turns out to be the so-called Mather-Jacobian multiplier ideals in the algebro-geometric setting. As applications, we obtain a reasonable generalization of (algebraic) adjoint ideal sheaves to the analytic setting and establish some extension theorems on K\"ahler manifolds from \emph{singular} hypersurfaces. Relying on our multiplier and adjoint ideals, we also give characterizations for several important classes of singularities of pairs associated to plurisubharmonic functions.

Moreover, we also investigate the local structure of singularities of log canonical locus of plurisubharmonic functions. Especially, in the three-dimensional case, we show that for any plurisubharmonic function with log canonical singularities, its associated multiplier ideal subscheme is weakly normal, by which we give a complete classification of multiplier ideal subschemes with log canonical singularities.
\end{abstract}

\maketitle
\tableofcontents

\section{Introduction}

The multiplier ideal sheaves together with a variant of them, the adjoint ideal sheaves, which measure the singularities of plurisubharmonic functions, turn out to be a powerful tool in complex geometry and algebraic geometry in recent years; one can refer to \cite{De10, G-Z_open, Siu05}$\&$\cite{La04, Kim15, Taka10} for a general exposition to the analytic and algebro-geometric side of the theory respectively. Throughout this article, all complex spaces are always assumed to be reduced and paracompact unless otherwise specified; we refer to \cite{GR84, KK83, Loja91, Richberg68} for main references on the theory of complex spaces.

In the present article, we firstly extend the concept of multiplier ideal sheaves to the singular case, by which, we then make the adjoint ideal sheaves to be well-defined in the analytic setting successfully and prove some interesting properties related to them. Since our multiplier ideals sheaves can be defined on any complex space of pure dimension (\emph{not} necessarily normal or $\mathbb{Q}$-Gorenstein), we also study the singularities of complex spaces together with plurisubharmonic functions on them by both of ideals; one can refer to \cite{L-Z} for a definition of multiplier ideal sheaves on $\mathbb{Q}$-Gorenstein complex spaces.

\subsection{Multiplier and adjoint ideal sheaves on complex spaces with singularities}

In \cite{Gue12}, Guenancia generalized the notion of adjoint ideal sheaves (along SNC divisors) to the analytic setting by means of the Ohsawa-Takegoshi-Manivel extension theorem, and then proved the coherence for the locally H\"older continuous plurisubharmonic weights (see also \cite{Kim15}). However, the analytic adjoint ideal sheaves defined in \cite{Gue12, Kim15} are \emph{not} coherent in general and one can find an explicit example given by Guan and the author in \cite{G-L_adjoint}. In the event that we would like to construct analytic adjoint ideals along arbitrary closed complex subspace of pure codimension as in the algebraic setting \cite{Taka10, Eisen10} and establish an analogous adjunction exact sequence, it is necessary for us to understand what is the meaning of multiplier ideals on complex spaces (may be singular), i.e., the last non-trivial term in the adjunction exact sequence.

Let $X$ be a complex space of pure dimension $n$, $\omega_X$ the \emph{dualizing sheaf} of $X$ and $\varphi\in L_{\text{loc}}^1(X_\text{reg})$ with respect to the Lebesgue measure. Then, we can define the \emph{Nadel-Ohsawa multiplier ideal sheaf} $\widehat{\mathscr{I}}(\varphi)\subset\mathscr{M}_X$ associated to the weight $\varphi$ on $X$ (see Definition \ref{MIS_NO}), via Ohsawa's extension measure. Moreover, we have the following:

\begin{theorem} \emph{(Theorem \ref{property}).}
Let $\varphi\in\emph{Psh}(X)$ be a plurisubharmonic function on $X$ such that $\varphi\not\equiv-\infty$ on every irreducible component of $X$. Then, $\widehat{\mathscr{I}}(\varphi)\subset\mathcal{O}_X$ is a coherent ideal sheaf and satisfies the strong openness property, i.e.,      $$\widehat{\mathscr{I}}(\varphi)=\widehat{\mathscr{I}}_+(\varphi):=\bigcup\limits_{\varepsilon>0}\widehat{\mathscr{I}}\big((1+\varepsilon)\varphi\big).$$
\end{theorem}

\begin{remark}
If the embedding dimension $\text{emb}_xX\leq n+1$ for all $x\in X$, an analogous result has been established in \cite{Li_adjoint}. When $X$ is smooth, $\widehat{\mathscr{I}}(\varphi)$ is nothing but the usual multiplier ideal sheaf $\mathscr{I}(\varphi)$.
\end{remark}

In the algebro-geometric setting, the authors in \cite{deFD14, EIM16} introduced the so-called Mather-Jacobian multiplier ideals via the Nash blow-up, which turned out to be very helpful in the study of higher-dimensional algebraic geometry. In fact, we can show that the analytic corresponding of Mather-Jacobian multiplier ideals is nothing but the multiplier ideals associated to analytic weights defined above, i.e.,

\begin{theorem} \emph{(Theorem \ref{MIS_equi}).}
Let $X$ be a complex space of pure dimension, $\mathfrak{a}\subset\mathcal{O}_X$ a nonzero ideal sheaf on $X$ and $c\in\mathbb{R}_{\geq0}$. Then, the Mather-Jacobian multiplier ideal sheaf associated to $\mathfrak{a}^c$ coincides with the Nadel-Ohsawa multiplier ideal sheaf associated to $\varphi_{c\cdot\frak{a}}$, i.e.,
$$\mathscr{I}_\emph{MJ}(\frak{a}^c)=\widehat{\mathscr{I}}(\varphi_{c\cdot\frak{a}}),$$
where $\varphi_{c\cdot\frak{a}}=\frac{c}{2}\log(\sum_k|f_k|^2)$ and $(f_k)$ is any local system of generators of $\frak{a}$.
\end{theorem}

Let $M$ be an $(n+r)$-dimensional complex manifold and $X\subset M$ a closed complex subspace of pure codimension $r$ with $g=(g_1,\dots,g_m)$ a system of generators of $\mathscr{I}_X$ near $x\in M$ ($m$ may depend on $x$). Let $\varphi\in\text{Psh}(M)$ be a plurisubharmonic function such that $\varphi|_X\not\equiv-\infty$ on every irreducible component of $X$. Then, we can obtain the following:

\begin{theorem} \emph{(Theorem \ref{adjoint}).} \label{Aadjoint}
There exists an ideal sheaf $$Adj_X(\varphi)\subset\mathcal{O}_M,$$
called the \emph{analytic adjoint ideal sheaf} associated to $\varphi$ along $X$, sitting in an exact sequence:
$$0\longrightarrow\mathscr{I}(\varphi+r\log|\mathscr{I}_X|)\stackrel{\iota}{\longrightarrow} Adj_X(\varphi)\stackrel{\rho}{\longrightarrow} i_*\widehat{\mathscr{I}}(\varphi|_X)\longrightarrow0, \eqno(\star)$$
where $i:X\hookrightarrow M,\ \iota$ and $\rho$ are the natural inclusion and restriction map respectively, and $\log|\mathscr{I}_X|:=\log|g|=\frac{1}{2}\log(|g_1|^2+\cdots+|g_m|^2)$ near every point $x\in M$.
\end{theorem}

\begin{remark}
$(1)$ Similar to the proof of Proposition 2.11 in \cite{Gue12}, it follows that $Adj_X(\varphi)$ coincides with the algebraic adjoint ideal sheaf defined by Takagi and Eisenstein in \cite{Taka10, Eisen10} whenever $\varphi$ has analytic singularities.

$(2)$ If $\varphi$ is a H\"older plurisubharmonic function and $X$ is a smooth divisor, our definition of $Adj_X(\varphi)$ is the same as that given by Guenancia (see Theorem 2.16 in \cite{Gue12}). Similarly, we always have $Adj_X(\varphi)\subset\mathscr{I}(\varphi)$, and $Adj_X(\varphi)_x=\mathscr{I}(\varphi)_x$ for each $x\not\in X$ as well.
\end{remark}

\subsection{Applications to extension theorems and singularities of pairs}

Using our notion of multiplier and adjoint ideal sheaves, we firstly establish the following extension theorem involved multiplier and adjoint ideals on K\"ahler manifolds from singular hypersufaces.

\begin{theorem} \label{vanishing_noncpt}
Let $(M,\omega)$ be a weakly pseudoconvex K\"ahler manifold and $H\subset M$ a (reduced) complex hypersurface. Let $L$ be a holomorphic line bundle on $M$ equipped with a (possibly singular) Hermitian metric $e^{-2\varphi_L}$ such that $\varphi_L|_H\not\equiv-\infty$ on each irreducible component of $H$ and $\sqrt{-1}\partial\bar\partial\varphi_L\geq\varepsilon\omega$ for some positive continuous function $\varepsilon$ on $M$. Then, the natural restriction map induces a surjection
$$H^0\big(M,\omega_M\otimes\mathcal{O}_M(L\otimes[H])\otimes Adj_H(\varphi_L)\big)\longrightarrow H^0\big(H,\omega_H\otimes\mathcal{O}_H(L|_H)\otimes\widehat{\mathscr{I}}(\varphi_L|_H)\big),$$
and
$$H^q\big(M,\omega_M\otimes\mathcal{O}_M(L\otimes[H])\otimes Adj_H(\varphi_L)\big)\widetilde{\longrightarrow}H^q\big(H,\omega_H\otimes\mathcal{O}_H(L|_H)\otimes\widehat{\mathscr{I}}(\varphi_L|_H)\big)$$
for all $q\geq1$.
\end{theorem}

\begin{remark}
In fact, we can further deduce $$H^q\big(M,\omega_M\otimes\mathcal{O}_M(L\otimes[H])\otimes Adj_H(\varphi_L)\big)=0$$ for all $q\geq1$, thanks to a singular version of Nadel vanishing theorem in \cite{De10}.
\end{remark}

Combining the above result and Siu's construction of metric, we obtain a singular version of Takayama's extension theorem in \cite{Taka06} as follows.

\begin{theorem} \emph{(Theorem \ref{Takayama}).}
Let $M$ be a smooth complex projective variety and $H\subset M$ a (reduced) hypersurface. Let $L$ be a holomorphic line bundle over $M$ equipped with a (possibly singular) Hermitian metric $e^{-2\varphi_L}$ such that $\sqrt{-1}\partial\bar\partial\varphi_L\geq\varepsilon\omega$ with $\varepsilon>0$ for some smooth Hermitian metric $\omega$ on $M$ and $\widehat{\mathscr{I}}(\varphi_L|_H)=\mathcal{O}_H$.

Then, the natural restriction map
$$H^0\big(M,\big(\omega_M\otimes\mathcal{O}_M(L\otimes[H])\big)^{\otimes m}\big)\longrightarrow H^0\big(H,\big(\omega_H\otimes\mathcal{O}_H(L)\big)^{\otimes m}\big)$$ is surjective for every $m > 0$.
\end{theorem}

Moreover, we can also establish the following singular version of extension theorem for projective family due to Siu \cite{Siu98, Siu02, Siu04} (see also \cite{Paun07}), by which we generalize Siu's theorem on plurigenera to the case of singular fibers with log terminal singularities (see Corollary \ref{Siu_invariance}).

\begin{theorem} \emph{(Theorem \ref{Siu_extension}).}
Let $\pi:M\to\Delta$ be a projective family and $L$ a holomorphic line bundle over $M$ endowed with a (possibly singular) Hermitian metric $e^{-2\varphi_L}$ such that $\sqrt{-1}\partial\bar{\partial}\varphi_L\geq0$. Assume that the restriction of $\varphi_L$ to the central fiber $M_0$ is well defined and $M_0$ has at most log terminal singularities.

Then, the natural restriction map
$$H^0\big(M,\omega_M^{\otimes m}\otimes\mathcal{O}_M(L)\big)\longrightarrow H^0\big(M_{0},\omega_{M_0}^{\otimes m}\otimes\mathcal{O}_{M_0}(L)\otimes\widehat{\mathscr{I}}(\varphi_L|_{M_0})\big)$$ is surjective for every $m > 0$.
\end{theorem}

Whereas the multiplier and adjoint ideals encode much information on the singularities of the underlying complex space together with the plurisubharmonic weights on them, we give the following characterization of singularities of pairs associated to plurisubharmonic functions.

\begin{theorem} \emph{(Theorem \ref{SingPairs}).} \label{SingPairs0}
Let $X$ be a complex space of pure dimension with $x\in X$ a point and $\varphi\in\emph{Psh}(X)$. Then,

$(1)$ If $(X,\varphi)$ is log terminal at $x$, then $x$ is a rational singularity of $X$.

$(2)$ If $(X,\varphi)$ is log canonical at $x$, then the complex subspace $(A,(\mathcal{O}_{X}/\widehat{\mathscr{I}}(\varphi))|_A)$ is reduced near $x$; in addition, if $\varphi$ has analytic singularities, then $(A,(\mathcal{O}_{X}/\widehat{\mathscr{I}}(\varphi))|_A)$ is weakly normal at $x$, where $A:=N(\widehat{\mathscr{I}}(\varphi))$ denotes the zero-set of coherent ideal sheaf $\widehat{\mathscr{I}}(\varphi)$.
\end{theorem}

When $X$ is locally a complete intersection, analogous to the characterization of rationality of hypersurface singularities in \cite{Li_adjoint}, we can establish the following

\begin{theorem} \label{rationality}
Let $X\subset M$ be locally a complete intersection and $\varphi\in\emph{Psh}(M)$ such that the slope $\nu_x(\varphi|_X)=0$ for every $x\in X$. Then, the analytic adjoint ideal sheaf $Adj_X(\varphi)=\mathcal{O}_M$ iff $\widehat{\mathscr{I}}(\varphi|_X)=\mathcal{O}_X$ iff $X$ is normal and has only rational singularities iff $(X,x)$ is canonical for any $x\in X$.
\end{theorem}

\begin{remark}
Note that our multiplier (resp. adjoint) ideals measure both the singularities of the plurisubharmonic weights and associated complex spaces (resp. subspaces) together. The above result also implies that if a locally complete intersection has at most rational singularities, the plurisubharmonic weights with zero slope do not add the singularities in the sense of multiplier or adjoint ideals.
\end{remark}

\begin{remark}
If $M$ is a smooth complex algebraic variety and $\varphi$ is trivial, the above result coincides with Proposition 9.3.48 (ii) in \cite{La04}, which turned out to be very useful to study the singularities of theta divisors on principally polarized abelian varieties.
\end{remark}

\subsection{Singularities of log canonical locus of plurisubharmonic functions}

Let $\Omega\subset\mathbb{C}^n$ be a domain with $o\in\Omega$ the origin and $u\in\text{Psh}(\Omega)$ a plurisubhamonic function on $\Omega$. The subscheme (or complex subspace) $V(u):=(A,(\mathcal{O}_\Omega/\mathscr{I}(u))|_A)$ of $\Omega$ cut out by the multiplier ideal sheaf $\mathscr{I}(u)$ is called a \emph{multiplier ideal subscheme} associated to $u$, where $A:=N(\mathscr{I}(u))$ denotes the zero-set of coherent ideal sheaf $\mathscr{I}(u)$; see e.g. \cite{De10, DK01, Nadel_behavior}.

The \emph{log canonical threshold} (or \emph{complex singularity exponent}) $c_{o}(u)$ or $\text{LCT}_o(u)$ of $u$ at $o$ is defined to be $$c_{o}(u):=\sup\{c\ge0\,|\,\exp(-2cu)\mbox{ is integrable near}\ o \}.$$
It is convenient to put $c_{o}(-\infty)=0$. If $c_{o}(u)=1$, we say that $(A,o)$ is the germ of \emph{log canonical locus} of $u$ at $o$.

Motivated by (2) of Theorem \ref{SingPairs0}, it is natural to raise the following question on the weak normality of log canonical locus of plurisubharmonic functions.

\begin{question} \label{Q_WNormal}
Let $o\in\Omega\subset\mathbb{C}^{n}$ be the origin, $u\in\text{\emph{Psh}}(\Omega)$ with $c_o(u)=1$. It is natural to wonder whether we are able to show that the multiplier ideal subscheme $V(u)$ is weakly normal at $o$.
\end{question}

In \cite{G-L_LCT16}, we answered the above Question affirmatively in dimension two. Owing to the absence of desingularization theorem for plurisubharmonic functions, we cannot deal with the above question like the algebraic situation, and so it is reasonable to find another plurisubharmonic weight with mild singularities whose multiplier ideal cuts out the same subscheme as $V(u)$. Thanks to Demailly's analytic approximation of plurisubharmonic functions via Bergman kernels, we can establish the following

\begin{theorem} \emph{(Theorem \ref{SiuPsh}).}
Let $\frak{a}_{\bullet}=\{\frak{a}_m\}$ be the graded system of ideal sheaves on $\Omega$ given by $m_\Lambda$-th symbolic powers $\frak{a}_m=\mathscr{I}_A^{<m_\Lambda>}$, where $m_k=e_km$ for each $k$ and $e_k$ is the codimension of $A_k$ in $\Omega$. Then, it follows that
$$\emph{LCT}_o(\frak{a}_{\bullet})=1\ \text{and}\ \mathscr{I}(\frak{a}_{\bullet})_o=\mathscr{I}(u)_o.$$
\end{theorem}

As an application of above result, we can confirm Question \ref{Q_WNormal} in dimension three, i.e.,

\begin{theorem} \label{3D}
Let $\Omega\subset\mathbb{C}^{3}$ be a domain containing the origin $o$ and $u\in\text{\emph{Psh}}(\Omega)$ with $c_o(u)=1$. Then, the multiplier ideal subscheme $V(u)$ is weakly normal at $o$.
\end{theorem}

\begin{remark}
As shown in Corollary \ref{LCI} and Corollary \ref{isolated}, the answer of Question \ref{Q_WNormal} is positive if $(A,o)$ is a complete
intersection, or $o$ is an isolated singularity of $A$.
\end{remark}

In view of Theorem \ref{3D}, we can in fact establish a strengthening of the statement as below.

\begin{corollary} \label{union}
With the same hypotheses as in Theorem \ref{3D}, we can derive that any union of irreducible components of $V(u)$ with the reduced complex structure is weakly normal near $o$.
\end{corollary}

\begin{remark}
By the same argument as in the proof of Corollary \ref{union}, we can obtain that if the answer of Question \ref{Q_WNormal} is positive, then any union of irreducible components of the multiplier ideal subscheme $V(u)$ is weakly normal near $o$.
\end{remark}

Furthermore, as a consequence of Theorem \ref{3D}, we are able to characterize the multiplier ideal subschemes with log canonical singularities as follows.

\begin{theorem} \label{char_3D}
Let $o\in\Omega\subset\mathbb{C}^{3}$ be the origin and $u\in\text{\emph{Psh}}(\Omega)$ with $c_{o}(u)=1$. Then, the multiplier ideal subscheme $V(u)$ at $o$ is exactly defined by one of the following ideals:

$(1)$ $(h)\cdot\mathcal{O}_{3}$, where $h\in\mathcal{O}_{3}$ is the minimal defining function of a germ of hypersurface with semi-log-canonical singularity at $o$;

$(2)$ $(z_1z_2,z_1z_3)\cdot\mathcal{O}_{3}$;

$(3)$ $(z_1,z_2)\cdot\mathcal{O}_{3}$;

$(4)$ $(z_1,z_2z_3)\cdot\mathcal{O}_{3}$;

$(5)$ $(z_1z_2,z_1z_3,z_2z_3)\cdot\mathcal{O}_{3}$;

$(6)$ $\mathfrak{m}_{3}$, up to a change of variables at $o$; where $\mathfrak{m}_{3}$ is the maximal ideal of $\mathcal{O}_{3}$.
\end{theorem}

\begin{remark}
For the sake of convenience, we provide a complete list on the classification of semi-log-canonical hypersurface singularities of dimension two in Table \ref{list} of Appendix \ref{B}; one can refer to \cite{K-SB} for more details (see also \cite{L-R}).

Moreover, all the cases in Theorem \ref{char_3D} will occur when we take $u$ to be the following weights respectively:
\begin{itemize}
  \item[(\text{1})] $u=\log|h|$;
  \item[(\text{2})] $u=\log(|z_1z_2^2|+|z_1z_3^2|)$;
  \item[(\text{3})] $u=\log(|z_1|^2+|z_2|^2)$;
  \item[(\text{4})] $u=\log(|z_1|^2+|z_2z_3|^2)$;
  \item[(\text{5})] $u=\log(|z_1z_2z_3|+|z_1z_2|^2+|z_1z_3|^2+|z_2z_3|^2)$;
  \item[(\text{6})] $u=\log(|z_1|^3+|z_2|^3+|z_3|^3)$.
\end{itemize}
\end{remark}

Based on Theorem \ref{char_3D}, in a subsequent paper \cite{G-L_LCT19}, we are able to answer the Question posed in \cite{G-L_LCT16} for the three-dimensional case, by combining the Ohsawa-Takegoshi $L^2$ extension theorem proved in \cite{OT87} with the restriction formula established in \cite{G-Z_restriction}. In particular, we can prove the following result in \cite{G-L_LCT19}:

\begin{theorem}
Let $o\in\Omega\subset\mathbb{C}^{3}\times\mathbb{C}^{n-3}\ (n\ge4)$ be the origin, $u\in\text{\emph{Psh}}(\Omega)$ and $H{=}\{z_4=\cdots=z_n=0\}\subset\Omega$ a three-dimensional plane through $o$. If $c_{o}(u|_H)=1$, then there exists a new local coordinates $(w_1,w_2,w_3;z_4,...,z_n)$ near $o$ such that $\mathscr{I}(u)_{o}$ is exactly equal to one of the following ideals:

$(1)\ \mathcal{O}_{n}$;\qquad\qquad\qquad $(2)\ \big(w_1,w_2,w_3\big)\cdot\mathcal{O}_{n}$;\qquad\qquad\qquad $(3)\ \big(w_1\big)\cdot\mathcal{O}_{n}$;

$(4)\ \big(w_1,w_2w_3+f(w_1,...,z_n)\big)\cdot\mathcal{O}_{n}$;

$(5)\ \big(w_1w_2,w_1w_3,w_2w_3\big)\cdot\mathcal{O}_{n}$;

$(6)\ \big(w_1w_2,w_1w_3\big)\cdot\mathcal{O}_{n}$;

$(7)\ \big(h(w_1,w_2,w_3)+f(w_1,...,z_n)\big)\cdot\mathcal{O}_{n}$, where $h\in\mathcal{O}_3$ is as in Theorem \ref{char_3D} and $f\in\mathscr{I}_{H,o}$.
\end{theorem}

Finally, relying on Theorem \ref{3D}, we give the following characterization of log canonical locus of pair $(X,\frak{a}_{\bullet})$ associated to graded system of ideals in the algebraic setting.

\begin{corollary} \label{LCS_pairs3D}
Let $X$ be a smooth complex quasi-projective threefold and $\frak{a}_{\bullet}=\{\frak{a}_m\}$ be any graded system of ideals on $X$.

If the log canonical threshold $\emph{LCT}(\frak{a}_{\bullet})=1$ on $X$, then the multiplier ideal subscheme $V(\frak{a}_{\bullet})$ associated to $\frak{a}_{\bullet}$ is weakly normal. In further, all the singularities of $V(\frak{a}_{\bullet})$ are precisely those presented in Theorem \ref{char_3D}.
\end{corollary}
\bigskip

\section{Multiplier and adjoint ideal sheaves}

In this section, we firstly present the Ohsawa's extension measure arising from the research of so-called Ohsawa-Takigoshi $L^2$ extension theorem. Next, we will construct the analytic multiplier and adjoint ideal sheaves on complex spaces via the measure and then prove some basic properties of them.

\subsection{Ohsawa's extension measure}

In order to establish a general $L^2$ extension theorem, Ohsawa \cite{Ohsawa5} introduced a positive measure on regular part of the associated closed complex subspace. Afterwards, associated with the same measure, Guan and Zhou \cite{G-Z_optimal} established two very general $L^2$ extension theorems with optimal estimates; see also \cite{De16, CDM17, Chan18} for various $L^2$ extension theorems related to Ohsawa's extension measure. In the subsequent parts, we will use the Ohsawa's extension measure to study the multiplier and adjoint ideal sheaves on complex spaces with singularities.
\medskip

First of all, let's recall some notations in \cite{Ohsawa5} (see also \cite{G-Z_optimal}). Let $M$ be an $(n+r)$-dimensional complex manifold and $X$ a (closed) complex subspace of $M$. Let $dV_{M}$ be a continuous volume form on $M$. Then, we consider a class of upper semi-continuous functions $\Psi$ from $M$ to the interval $[-\infty,A)$ with $A\in(-\infty,+\infty]$, such that

$(1)$ $\Psi^{-1}(-\infty)\supset X$ and

$(2)$ If $X$ is $k$-dimensional around a point $x\in X\mspace{0mu}_{\text{reg}}$ ( the regular part of $X$), there exists a local coordinate $(z_{1},\cdots,z_{n+r})$ on a neighborhood $U$ of $x$ such that $z_{k+1}=\cdots=z_{n+r}=0$ on $X\cap U$ and
$$\sup_{U\setminus X}\big|\Psi(z)-(n+r-k)\log\sum_{j=k+1}^{n+r}|z_{j}|^{2}\big|<+\infty.$$

The set of such functions $\Psi$ will be denoted by $\#_{A}(X)$. For each $\Psi\in\#_{A}(X)$, one can associate a positive measure $dV_{X}[\Psi]$ on $X_\text{reg}$ as the minimum element of the partially ordered set of positive measures $d\mu$ satisfying
$$\int_{X_{k}}fd\mu\geq\limsup_{t\to+\infty}\frac{2(n+r-k)}{\sigma_{2(n+r-k)-1}}\int_{M}fe^{-\Psi}\mathbbm{1}_{\{-t-1<\Psi<-t\}}dV_{M}$$
for any nonnegative continuous function $f$ with $\text{Supp}\,f\,{\subset\subset}\,M\backslash X_\text{sing}$, where $X_{k}$ denotes the $k$-dimensional component of $X_\text{reg}$, $\sigma_{m}$ denotes the volume of the unit sphere in $\mathbb{R}^{m+1}$ and $\mathbbm{1}_{\{-t-1<\Psi<-t\}}$ denotes the characteristic function of the set $\{z\in M|-t-1<\Psi(z)<-t\}$.

\begin{remark} \label{computation}
If $X\subset M$ is a complex subspace of pure codimension $r$ such that $\mathscr{I}_X$ is globally generated by holomorphic functions $g_1,\dots,g_m$ on $M$, by taking $$\Psi=r\log(|g_1|^2+\cdots+|g_m|^2),$$ one can check that the measure $dV_{X}[\Psi]$ on $X_\text{reg}$ can be defined by
$$\int_{X_\text{reg}}fdV_{X}[\Psi]=\limsup_{t\to+\infty}\frac{2r}{\sigma_{2r-1}}\int_{M}fe^{-\Psi}\mathbbm{1}_{\{-t-1<\Psi<-t\}}dV_{M}$$
for any nonnegative continuous function $f$ with $\text{Supp}\,f\,{\subset\subset}\,M\backslash X_\text{sing}$.
\medskip

By Theorem 2.0.2 in \cite{Wlo08} (see also \cite{Eisen10}, Corollary 3.2), we obtain a strong factorizing desingularization $\pi:\widetilde M\to M$ of $X$ such that $\pi^*(\mathscr{I}_X)=\mathscr{I}_ {\widetilde X}\cdot\mathscr{I}_{R_X}$, where $\widetilde X$ is the strict transform of $X$ in $\widetilde M$ and $R_X$ is an effective divisor supported on the exceptional divisor $\text{Ex}(\pi)$ of $\pi$. Then, one can check that the measure $dV_{X}[\Psi]$ is the direct image of measures defined upstairs by
$$f\longmapsto\int_{\widetilde X}\left|f\circ\pi|_{\widetilde X}\right|^2\cdot\left|\frac{\text{Jac}(\pi)}{h_{R_X}^r}\Big|_{\widetilde X}\right|^2dV_{\widetilde X},$$
up to equivalence.
\medskip

In particular, if $M$ is a domain in $\mathbb{C}^{n+r}$ and $X\subset M$ is a complete intersection, then one can check that $dV_{X}[\Psi]$ is the measure on $X_{\text{reg}}$ such that $$dV_{X}[\Psi]=\frac{1}{|\Lambda^r(dg)|^2}dV_X,$$
where $$dV_X=\frac{1}{n!}\omega^{n}|_{X_\text{reg}}\ \ \text{and}\ \ \omega=\frac{\sqrt{-1}}{2}\sum\limits_{k=1}^{n+r}dz_k\wedge d\bar z_k.$$
\end{remark}

\subsection{Multiplier ideal sheaves on complex spaces}

Different from the algebraic multiplier ideals, the usual analytic multiplier ideals are constructed by the integrability of multipliers associated to plurisubharmonic weights with respect to the Lebesgue measure. In the analytic setting, the main difficulty of extending the notion of multiplier ideals to the singular case is how to choose a suitable measure in the construction by means of integrability.

For $\mathbb{Q}$-Gorenstein complex spaces (e.g., complete intersections), it is natural to choose an adapted measure through the associated pluricanonical forms (cf. \cite{L-Z}). Unfortunately, the approach is out of work for the general case. In the following, we will construct the multiplier ideals on complex spaces of pure dimension, via the Ohsawa's extension measure.

\begin{definition} \label{MIS_NO}
Let $X$ be a complex space of pure dimension $n$ and $\varphi\in L_\text{loc}^1(X_\text{reg})$ with respect to the Lebesgue measure. Let $X\hookrightarrow\Omega\subset\mathbb{C}^{n+r}$ be a local embedding such that $\mathscr{I}_X$ is generated by holomorphic functions $g_1,\dots,g_m$ on $\Omega$.

The \emph{Nadel-Ohsawa multiplier ideal sheaf} $\widehat{\mathscr{I}}(\varphi)$ on $X$ is defined to be the fractional ideal sheaf of germs of meromorphic functions $f\in\mathscr{M}_{X,x}$ such that $|f|^2e^{-2\varphi}$ is locally integrable at $x$ on $X$ with respect to the measure $dV_{X}[\Psi]$, where $\Psi=r\log(|g_1|^2+\cdots+|g_m|^2)$. One can check that $\widehat{\mathscr{I}}(\varphi)$ is independent of the local embedding of $X$ and the choice of generators of $\mathscr{I}_{X}$.
\end{definition}

\begin{remark}
$(1)$ By a suitable choice of the polar function $\Psi$, we can also define multiplier ideal sheaves on complex spaces \emph{not} necessarily of pure dimension.

$(2)$ If the complex space $X\subset M$ is locally a complete intersection (not necessarily normal), thanks to the adjunction formula, it yields that $\omega_X^{\text{GR}}(\varphi)=\omega_X\otimes\widehat{\mathscr{I}}(\varphi)$, where $\omega_X^{\text{GR}}(\varphi)$ is the $\mathcal{O}_X$-sheaf on $X$ defined by
$$\Gamma(U,\omega_X^{\text{GR}}(\varphi))=\{\sigma\in\Gamma(U\cap X_\text{reg},\omega_{X_\text{reg}})\ |\
(\sqrt{-1})^{n^2}e^{-2\varphi}\sigma\wedge\overline\sigma\in L_\text{loc}^1(U)\}$$
for any open subset $U\subset X$.
\end{remark}

Analogous to the smooth case, we have the following:

\begin{theorem} \label{property}
With the same notations as above and $\varphi\in\emph{Psh}(X)$, it follows that

$(1)$ \emph{(Coherence).} $\widehat{\mathscr{I}}(\varphi)\subset\mathcal{O}_X$ is a coherent ideal sheaf.

$(2)$ \emph{(Strong openness).}
      $\widehat{\mathscr{I}}(\varphi)=\widehat{\mathscr{I}}_+(\varphi):=\bigcup\limits_{\varepsilon>0}\widehat{\mathscr{I}}\big((1+\varepsilon)\varphi\big)$.
\end{theorem}

For the proof, we may assume that $X$ is a complex subspace of some domain $\Omega$ in $\mathbb{C}^{n+r}$ and $\pi:\widetilde\Omega\to\Omega$ is a strong factorizing desingularization of $X$ such that $\pi^*(\mathscr{I}_X)=\mathscr{I}_ {\widetilde X}\cdot\mathscr{I}_{R_X}$, where $\widetilde X$ is the strict transform of $X$ in $\widetilde\Omega$ and $R_X$ is an effective divisor supported on $\text{Ex}(\pi)$. Then, we can show that
$$\widehat{\mathscr{I}}(\varphi)=\pi_*\left(\mathcal{O}(K_{\widetilde\Omega/\Omega}-rR_X)|_{\widetilde X}\otimes\mathscr{I}(\varphi\circ\pi|_{\widetilde X})\right).$$

\begin{remark} \label{Re_SOC}
In fact, for the strong openness of multiplier ideal sheaves, if $(\varphi_k)$ is a sequence of plurisubharmonic functions converging to $\varphi$ with $\varphi_k\leq\varphi$ on $X$, then we have
$$\widehat{\mathscr{I}}(\varphi)=\widehat{\mathscr{I}}_+(\varphi)=\bigcup\limits_{k}\widehat{\mathscr{I}}(\varphi_k).$$
\end{remark}

\begin{remark}
If $\varphi_A$ is a (quasi)-plurisubharmonic function on $X$ with analytic singularities and $\varphi_A\not\equiv-\infty$ on every irreducible component of $X$, by combining with an argument of log resolution, we can deduce that the fractional ideal sheaf $\widehat{\mathscr{I}}(\varphi-\varphi_A)\subset\mathscr{M}_X$ is coherent and satisfies the strong openness property (cf. \cite{L-Z}).
\end{remark}

Combining with Lemma 4.4 in \cite{Eisen10}, we achieve the equivalence of Mather-Jacobian and Nadel-Ohsawa multiplier ideal sheaves for analytic weights as follows.

\begin{theorem} \label{MIS_equi}
Let $X$ be a complex space of pure dimension, $\mathfrak{a}\subset\mathcal{O}_X$ a nonzero ideal sheaf on $X$ and $c\in\mathbb{R}_{\geq0}$. Then, the Mather-Jacobian multiplier ideal sheaf associated to $\mathfrak{a}^c$ coincides with the Nadel-Ohsawa multiplier ideal sheaf associated to $\varphi_{c\cdot\frak{a}}$, i.e.,
$$\mathscr{I}_\emph{MJ}(\frak{a}^c)=\widehat{\mathscr{I}}(\varphi_{c\cdot\frak{a}}),$$
where $\varphi_{c\cdot\frak{a}}=\frac{c}{2}\log(\sum_k|f_k|^2)$ and $(f_k)$ is any local system of generators of $\frak{a}$.
\end{theorem}

\begin{remark} \label{measure_equi}
In fact, similar to the argument in the proof of Theorem \ref{MIS_equi}, we can derive the following characterization of Ohsawa's extension measure.

Let $(M,\omega)$ be a Hermitian manifold and $X\subset M$ a (closed) complex subspace of pure dimension $d$ with the volume element $dV_{X,\omega}=\frac{1}{d!}\omega^{d}|_{X_\emph{reg}}$. Let $\Psi\in\#_A(X)$ be a polar function on $M$ determined by the ideal sheaf $\mathscr{I}_X$. Then, there exists a locally bounded positive function $\rho$ on $X$ such that
$$dV_{X,\omega}[\Psi]=\rho\cdot\frac{dV_{X,\omega}}{|\mathcal{J}ac_X|^2},$$
where $|\mathcal{J}ac_X|^2$ is globally defined via a partition of unity on the coordinate charts covering of $X$.
\end{remark}

\subsection{Adjoint ideal sheaves along closed complex subspaces}

In \cite{Gue12}, Guenancia defined an analytic adjoint ideal sheaf associated to plurisubharmonic weight $\varphi$ along SNC divisors (but not coherent for general $\varphi$; see \cite{G-L_adjoint} for more details), and established the so-called adjunction exact sequence for the locally H\"older continuous plurisubharmonic weights and smooth divisors. Later, the author \cite{Li_adjoint} obtained a well-defined generalization of adjoint ideal sheaves to the analytic setting in the case of divisors. In this subsection, owing to our multiplier ideals introduced as above, we will make a reasonable version of the definition of analytic adjoint ideal sheaves along closed complex subspaces of higher codimension.

\begin{theorem} \label{adjoint} \emph{(Theorem \ref{Aadjoint}).}
Let $M$ be an $(n+r)$-dimensional complex manifold and $X\subset M$ a closed complex subspace of pure codimension $r$ with $g=(g_1,\dots,g_m)$ a system of generators of $\mathscr{I}_X$ near $x\in M$ ($m$ may depend on $x$) and $\varphi\in\emph{Psh}(M)$ such that $\varphi|_X\not\equiv-\infty$ on every irreducible component of $X$.

Then, there exists an ideal sheaf $$Adj_X(\varphi)\subset\mathcal{O}_M,$$
called the \emph{analytic adjoint ideal sheaf} associated to $\varphi$ along $X$, sitting in an exact sequence:
$$0\longrightarrow\mathscr{I}(\varphi+r\log|\mathscr{I}_X|)\stackrel{\iota}{\longrightarrow} Adj_X(\varphi)\stackrel{\rho}{\longrightarrow} i_*\widehat{\mathscr{I}}(\varphi|_X)\longrightarrow0, \eqno(\star)$$
where $i:X\hookrightarrow M,\ \iota$ and $\rho$ are the natural inclusion and restriction map respectively, and $$\log|\mathscr{I}_X|:=\log|g|=\frac{1}{2}\log(|g_1|^2+\cdots+|g_m|^2)$$ near every point $x\in X$.
\end{theorem}

\noindent{\emph{Sketch of the proof}.} As the statement is local, without loss of generality, we may assume that $M$ is a bounded Stein domain in $\mathbb{C}^{n+r}$, $g=(g_1,\dots,g_m)$ is a system of generators of $\mathscr{I}_X$.\\

\textbf{Case (i).} \emph{When $X\subset M$ is a closed complex submanifold of pure codimension $r$.}

Let $\mathscr{J}\subset\mathcal{O}_M$ be an ideal sheaf such that $\mathscr{J}|_X=\widehat{\mathscr{I}}(\varphi|_X)$, which implies that $\mathscr{J}+\mathscr{I}_X$ is independent of the choice of $\mathscr{J}$. Set $$Adj_X(\varphi):=\bigcup\limits_{\varepsilon>0}Adj_X^0\Big((1+\varepsilon)\varphi\Big)\cap\Big(\mathscr{J}+\mathscr{I}_X\Big),$$ where $Adj_X^0(\varphi)\subset \mathcal{O}_M$ is an ideal sheaf of germs of holomorphic functions $f\in\mathcal{O}_{M,x}$ such that $$\frac{|f|^{2}e^{-2\varphi}}{|g|^{2r}\log^2|g|}$$ is locally integrable with respect to the Lebesgue measure near $x$ on $M$. \\

\textbf{Case (ii).} \emph{When $X\subset M$ is a closed complex subspace with singularities.}

Let $\pi:\widetilde M\to M$ be a strong factorizing desingularization of $X$ such that $\pi^*(\mathscr{I}_X)=\mathscr{I}_ {\widetilde X}\cdot\mathscr{I}_{R_X}$, where $\widetilde X$ is the strict transform of $X$ in $\widetilde M$ and $R_X$ is an effective divisor supported on $\text{Ex}(\pi)$. Thanks to Case (i), we have the following adjunction exact sequence
$$0\longrightarrow\mathscr{I}(\varphi\circ\pi+r\log|\mathscr{I}_{\widetilde X}|)\longrightarrow Adj_{\widetilde X}(\varphi\circ\pi)\longrightarrow i_*\widehat{\mathscr{I}}(\varphi\circ\pi|_{\widetilde X})\longrightarrow0.$$
Twist the exact sequence by $\mathcal{O}_{\widetilde M}(K_{\widetilde M/M}-rR_X)$, and then we deduce that
\begin{equation*}
\begin{split}
0\longrightarrow\mathcal{O}_{\widetilde M}(K_{\widetilde M/M})\otimes\mathscr{I}(\varphi\circ\pi+r\log|\mathscr{I}_{X}|\circ\pi)&\longrightarrow Adj_{\widetilde X}(\varphi\circ\pi)\otimes\mathcal{O}_{\widetilde M}(K_{\widetilde M/M}-rR_X)\\
&\longrightarrow i_*\mathscr{I}(\varphi\circ\pi|_{\widetilde X})\otimes\mathcal{O}_{\widetilde M}(K_{\widetilde M/M}-rR_X)\longrightarrow0.
\end{split}
\end{equation*}
Then, we can take $$Adj_X(\varphi):=\pi_*\left(Adj_{\widetilde X}(\varphi\circ\pi)\otimes\mathcal{O}_{\widetilde M}(K_{\widetilde M/M}-rR_X)\right).$$
Here, in order to establish the exact sequence $(\star)$, a local vanishing theorem is necessary (cf. \cite{Matsu_morphism}, Corollary 1.5).

\begin{remark}
$(1)$ The adjunction exact sequence $(\star)$ yields the coherence and uniqueness (independent of the desingularization $\pi$) of $Adj_X(\varphi)$.

$(2)$ If $X\subset M$ is a complex hypersurface, we can also get the following short exact sequence by a similar discussion as in Proposition 9.3.48 in \cite{La04}:
$$0\longrightarrow\mathcal{O}_M(K_M)\longrightarrow\mathcal{O}_M(K_M+X)\otimes Adj_X(\varphi)\longrightarrow\pi_*\mathcal{O}_{\widetilde X}(K_{\widetilde{X}})\longrightarrow0,$$
where $\pi:\widetilde M\to M$ is an embedded resolution of $(M,X)$ such that the proper transform $\widetilde X\subset\widetilde M$ of $X$ is non-singular.

$(3)$ When $M$ is an algebraic variety over $\mathbb{C}$, $X\subset M$ is a $\mathbb{Q}$-Gorenstein reduced equidimensional subscheme and $\varphi$ has analytic singularities, the above result is nothing but the main theorem (Theorem 5.1) in \cite{Eisen10}, which answered a question of Takagi in \cite{Taka10}.
\end{remark}

\subsection{The restriction theorem of multiplier ideals}

This part is devoted to discuss the relation of multiplier ideal sheaf $\widehat{\mathscr{I}}(\varphi|_X)$ and the restriction of $Adj_X(\varphi)$ and $\mathscr{I}(\varphi)$ to $X$. Concretely, we can prove the following:

\begin{theorem} \label{restriction}
With hypotheses as above, it follows that
$$\widehat{\mathscr{I}}(\varphi|_X)\subset\mathscr{I}\big(\varphi+(r-\delta)\log|g|\big)\cdot\mathcal{O}_X$$ for any $0<\delta\leq r$;
and
$$\widehat{\mathscr{I}}(\varphi|_X)=Adj_X(\varphi)\cdot\mathcal{O}_X.$$
\end{theorem}

\begin{proof}
The first inclusion is a direct consequence of Theorem \ref{property} (strong openness) and a local version of Theorem 2.2 in \cite{G-Z_optimal}, with weights $(1+\varepsilon)\varphi$.

Twist the exact sequence $(\star)$ through by $\mathcal{O}_X$, and then we deduce $Adj_X(\varphi)\cdot\mathcal{O}_X=\widehat{\mathscr{I}}(\varphi|_X)$ from the right exactness of tensor functor.
\end{proof}

\begin{remark}
Note that $Adj_X(\varphi)\subset\mathscr{I}\big(\varphi+(r-\delta)\log|g|\big)$ for any $0<\delta\leq r$ by the definition of $Adj_X(\varphi)$. Then, we can also infer the first inclusion from $\widehat{\mathscr{I}}(\varphi|_X)=Adj_X(\varphi)\cdot\mathcal{O}_X$.
\end{remark}
\bigskip

\section{Extension theorems on K\"ahler manifolds from singular hypersurfaces}

\subsection{Extension theorems for adjoint ideals on K\"ahler manifolds}

In this subsection, we turn to obtain a global extension theorem related to adjoint ideals on K\"ahler manifolds from singular hypersuraces by combining the adjunction exact sequence $(\star)$ with the Nadel vanishing theorem.\medskip

\noindent{\textbf{\emph{Proof of Theorem} \ref{vanishing_noncpt}.}} It follows from the adjunction formula for the hypersurface $H$ that $\omega_H=\big(\omega_M\otimes\mathcal{O}_M([H])\big)|_H$. Then, we consider the long exact sequence of cohomology associated to the short exact sequence $(\star)$ twisted by $\omega_M\otimes\mathcal{O}_M(L\otimes[H])$, i.e.,
\begin{equation*}
\begin{split}
0&\longrightarrow\omega_M\otimes\mathcal{O}_M(L)\otimes\mathscr{I}(\varphi)\longrightarrow\omega_M\otimes\mathcal{O}_M(L\otimes[H])\otimes Adj_H(\varphi)\\
&\longrightarrow i_*\big(\omega_H\otimes\mathcal{O}_H(L|_H)\otimes\widehat{\mathscr{I}}(\varphi|_H)\big)\longrightarrow0.
\end{split}
\end{equation*}
Thereupon, the desired result follows immediately from the Nadel vanishing theorem.
\hfill $\Box$\\

In addition, if $M$ is compact and $(L,e^{-2\varphi})$ is only a pseudo-effective line bundle (not necessarily big), by replacing the Nadel vanishing theorem in the proof of above theorem by a Kawamata--Viehweg--Nadel-type vanishing theorem established in \cite{Cao, G-Z_open}, i.e., $$H^q\big(M,\omega_M\otimes\mathcal{O}_M(L)\otimes\mathscr{I}(\varphi)\big)=0$$
for all $q\geq n-\text{nd}(L,\varphi)+1$, where $\text{nd}(L,\varphi)$ is the numerical dimension of $(L,\varphi)$ as defined in \cite{Cao}, we can obtain the following:

\begin{theorem} \label{vanishing_cpt}
Let $M$ be a compact K\"ahler manifold of dimension $n$ and $H\subset M$ a (reduced) complex hypersurface. Let $(L,e^{-2\varphi})$ be a pseudo-effective line bundle on $M$ such that $\varphi|_H\not\equiv-\infty$ on each irreducible component of $H$. Then, the natural restriction map induces a surjection
$$H^{q-1}\big(M,\omega_M\otimes\mathcal{O}_M(L\otimes[H])\otimes Adj_H(\varphi)\big)\longrightarrow H^{q-1}\big(H,\omega_H\otimes\mathcal{O}_H(L|_H)\otimes\widehat{\mathscr{I}}(\varphi|_H)\big),$$
and
$$H^q\big(M,\omega_M\otimes\mathcal{O}_M(L\otimes[H])\otimes Adj_H(\varphi)\big)\widetilde{\longrightarrow}H^q\big(H,\omega_H\otimes\mathcal{O}_H(L|_H)\otimes\widehat{\mathscr{I}}(\varphi|_H)\big)$$
for all $q\geq n-\emph{nd}(L,\varphi)+1$.
\end{theorem}

\begin{remark}
In particular, if $H\subset M$ is a smooth complex hypersurface in Theorem \ref{vanishing_cpt}, then we have
$$H^q\big(M,\omega_M\otimes\mathcal{O}_M(L\otimes[H])\otimes Adj_H(\varphi)\big)=0$$
for all $q\geq n-\text{nd}(L,\varphi)+1$, by the corresponding vanishing theorems.
\end{remark}

Combining Theorem \ref{vanishing_noncpt} with Siu's construction of the metric in \cite{Siu98, Siu04} (see also \cite{Varolin08_extension}), we can derive the following generalization of Takayama's extension theorem from singular hypersurface.

\begin{theorem} \emph{(Takayama's extension theorem, \cite{Taka06}).} \label{Takayama}
Let $M$ be a smooth complex projective variety and $H\subset M$ a (reduced) hypersurface. Let $L$ be a holomorphic line bundle over $M$ equipped with a (possibly singular) Hermitian metric $e^{-2\varphi_L}$ such that $\sqrt{-1}\partial\bar\partial\varphi_L\geq\varepsilon\omega$ with $\varepsilon>0$ for some smooth Hermitian metric $\omega$ on $M$ and $\widehat{\mathscr{I}}(\varphi_L|_H)=\mathcal{O}_H$.

Then, the natural restriction map
$$H^0\big(M,\big(\omega_M\otimes\mathcal{O}_M(L\otimes[H])\big)^{\otimes m}\big)\longrightarrow H^0\big(H,\big(\omega_H\otimes\mathcal{O}_H(L)\big)^{\otimes m}\big)$$ is surjective for every $m > 0$.
\end{theorem}

\subsection{Siu's extension theorem and deformation invariance of plurigenera}

Let $M$ be a complex manifold and $\pi:M\to\Delta=\{t\in\mathbb{C}\,|\,|t|<1\}$ a proper surjective holomorphic map with reduced analytic fibers $M_t=\pi^{-1}(t)$. The holomorphic family $\pi:M\to\Delta$ is called \emph{projective} if there is a positive holomorphic line bundle over $M$. Then, we are able to establish the following singular version of extension theorem for projective family due to Siu \cite{Siu98, Siu02, Siu04} (see also \cite{Paun07}),

\begin{theorem} \label{Siu_extension}
Let $\pi:M\to\Delta$ be a projective family and $L$ a holomorphic line bundle over $M$ endowed with a (possibly singular) Hermitian metric $e^{-2\varphi_L}$ such that $\sqrt{-1}\partial\bar{\partial}\varphi_L\geq0$. Assume that the restriction of $\varphi_L$ to the central fiber $M_0$ is well defined and $M_0$ has at most log terminal singularities.

Then, the natural restriction map
$$H^0\big(M,\omega_M^{\otimes m}\otimes\mathcal{O}_M(L)\big)\longrightarrow H^0\big(M_{0},\omega_{M_0}^{\otimes m}\otimes\mathcal{O}_{M_0}(L)\otimes\widehat{\mathscr{I}}(\varphi_L|_{M_0})\big)$$ is surjective for every $m > 0$.
\end{theorem}

Analogous to the proof of Theorem \ref{Siu_extension}, we deduce the following extension theorem for projective family with singular fibers; see \cite{Claudon07} for an argument on the case of smooth fibers.

\begin{corollary}
Let $\pi:M\to\Delta$ be a projective family and $(L,h_L)$ a holomorphic line bundle over $M$ endowed with a (possibly singular) Hermitian metric $e^{-2\varphi_L}$ such that the curvature current $\sqrt{-1}\partial\bar{\partial}\varphi_L\geq0$. Assume that the restriction of $\varphi_L$ to the central fiber $M_0$ is well defined and $\widehat{\mathscr{I}}(\varphi_L|_{M_0})=\mathcal{O}_{M_0}$.

Then, the natural restriction map
$$H^0\big(M,\big(\omega_M\otimes\mathcal{O}_M(L)\big)^{\otimes m}\big)\longrightarrow H^0\big(M_{0},\big(\omega_{M_0}\otimes\mathcal{O}_{M_0}(L)\big)^{\otimes m}\big)$$ is surjective for every $m > 0$.
\end{corollary}

As a straightforward application of Theorem \ref{Siu_extension}, we obtain the following singular version of deformation invariance of plurigenera established by Siu in \cite{Siu98, Siu02}:

\begin{corollary} \emph{(Siu's theorem on plurigenera).} \label{Siu_invariance}
Let $\pi:M\to\Delta$ be a projective family and assume that every fiber $M_t$ has at most log terminal singularities. Then for each $m\geq0$, the plurigenus
$$P_m(M_t):=\dim_{\mathbb{C}}H^0(M_t,\omega_{M_t}^{\otimes m})$$
is independent of $t$.
\end{corollary}

\begin{remark}
In fact, in our situation, we have the equivalence of log terminal singularities and canonical singularities by Theorem \ref{rationality}. Thus, the above invariance of plurigenera has been established in \cite{Taka07} by combining a complicated $L^2$ extension theorem with some algebraic techniques.
\end{remark}
\bigskip

\section{Singularities of pairs associated to plurisubharmonic functions}

In this section, we study the singularities of complex spaces (\emph{not} necessarily normal or $\mathbb{Q}$-Gorenstein). Concretely, we will characterize several important classes of singularities of pairs associated to plurisubharmonic functions in terms of the triviality of our multiplier and adjoint ideals. We refer to \cite{deFD14, EIM16} and \cite{Ambro11, Kollar_MMP} for some algebraic counterparts on singularities defined by discrepancies.

\begin{definition}
Let $X$ be a complex space of pure dimension with $x\in X$ a point and $\varphi\in\text{Psh}(X)$ a plurisubharonic function on $X$.

The pair $(X,\varphi)$ is \emph{log terminal} at $x$ if $\widehat{\mathscr{I}}(\varphi)_x=\mathcal{O}_{X,x}$. The point $x$ is called a \emph{log terminal singularity} of $X$ if trivial pair $(X,0)$ is log terminal at $x$.

The pair $(X,\varphi)$ is \emph{log canonical} at $x$ if $\widehat{\mathscr{I}}((1-\varepsilon)\cdot\varphi)_x=\mathcal{O}_{X,x}$ for all $0<\varepsilon<1$. $(X,\varphi)$ is \emph{log terminal} (resp. \emph{log canonical}) if it is log terminal (resp. log canonical) at every point in $X$.
\end{definition}

\begin{theorem} \label{SingPairs}
Let $X$ be a complex space of pure dimension with $x\in X$ a point and $\varphi\in\emph{Psh}(X)$. Then,

$(1)$ If $(X,\varphi)$ is log terminal at $x$, then $x$ is a rational singularity of $X$.

$(2)$ If $(X,\varphi)$ is log canonical at $x$, then the complex subspace $(A,(\mathcal{O}_{X}/\widehat{\mathscr{I}}(\varphi))|_A)$ is reduced near $x$; in addition, if $\varphi$ has analytic singularities, then $(A,(\mathcal{O}_{X}/\widehat{\mathscr{I}}(\varphi))|_A)$ is weakly normal at $x$, where $A:=N(\widehat{\mathscr{I}}(\varphi))$ denotes the zero-set of coherent ideal sheaf $\widehat{\mathscr{I}}(\varphi)$.
\end{theorem}

For our proof of above result, we need the following local vanishing theorem, which is a straightforward consequence of Corollary 1.5 in \cite{Matsu_morphism} and the arguments as in Theorem 3.5 in \cite{EIM16}.

\begin{theorem}
Let $X$ be a complex space of pure dimension and $\pi:\widetilde X\to X$ a log resolution of $\mathcal{J}ac_X$ such that $\mathcal{J}ac_X\cdot\mathcal{O}_{\widetilde X}=\mathcal{O}_{\widetilde X}(-J_{\widetilde X/X})$. Let $(L,e^{-2\varphi_L})$ be a (possibly singular) Hermitian line bundle on $\widetilde X$ with semi-positive curvature current. Then,
$$R^q\pi_*\left(\mathcal{O}_{\widetilde X}(\widehat K_{\widetilde X/X}-J_{\widetilde X/X})\otimes\mathcal{O}_X(L)\otimes\mathscr{I}(\varphi_L)\right)=0,$$
for every $q>0$.
\end{theorem}

Our proof of (1) in Theorem \ref{SingPairs} depends on an induction on $\dim X$ due to a Bertini property related with log terminal singularities (cf. Lemma 4.1 in \cite{Greb11}) and (2) is an immediate result of strong openness of multiplier ideal sheaves and the above local vanishing theorem.
\bigskip

Moreover, analogous to the argument of the case of hypersurfaces as in \cite{Li_adjoint}, we can establish the following

\begin{theorem}
Let $X\subset M$ be locally a complete intersection and $\varphi\in\emph{Psh}(M)$ such that the slope $\nu_x(\varphi|_X)=0$ for every $x\in X$. Then, the analytic adjoint ideal sheaf $Adj_X(\varphi)=\mathcal{O}_M$ iff $\widehat{\mathscr{I}}(\varphi|_X)=\mathcal{O}_X$ iff $X$ is normal and has only rational singularities iff $(X,x)$ is canonical for any $x\in X$.
\end{theorem}

Here, the slope of $\varphi|_X$ at $x$ is defined by
$$\nu_x(\varphi|_X):=\sup\{\gamma\geq0\ \big|\ \varphi|_X\leq\gamma\log\sum_k|f_k|+O(1)\}\in[0,+\infty),$$
where $(f_k)$ are local generators of the maximal ideal $\frak{m}_x$ of $\mathcal{O}_{X,x}$ (cf. \cite{BBEGZ}).
\bigskip

\section{Log canonical singularities of plurisubharmonic functions}

This section is devoted to investigate the local structure of log canonical singularities of plurisubharmonic functions (see \cite{De16, CDM17, Chan18} for some global properties related to various $L^2$ extension theorems).

\subsection{Finite generation of graded system of ideal sheaves}

In this subsection, we study the finite generation of certain graded system of ideal sheaves from symbolic powers; see \cite{Kollar_MMP, La04} for the basic references. Firstly, let us recall some related concepts and notations.

\begin{definition}
Let $X$ be a complex manifold and $Z\subset X$ an analytic set with ideal sheaf $\mathscr{I}_Z\subset\mathcal{O}_X$. Let $Z=\bigcup_{k\in\Lambda}Z_k$ be a global irreducible decomposition of $Z$ and $m_\Lambda:=(m_k)_{k\in\Lambda}$ be a $\Lambda$-tuple with $m_k\in\mathbb{N}$.

The \emph{$m_\Lambda$-th symbolic power} $\mathscr{I}_Z^{<m_\Lambda>}$ of $\mathscr{I}_Z$ is the ideal sheaf of germs of holomorphic functions that have multiplicity $\geq m_k$ at a general point of $Z_k$ for each $k\in\Lambda$, i.e.,
$$\mathscr{I}_Z^{<m_\Lambda>}:=\{f\in\mathcal{O}_X\ \big|\ \text{ord}_xf\geq m_k\text{\ for general point $x$ of each $Z_k$}\}.$$
\end{definition}

In particular, if $m_k=m$ for every $k\in\Lambda$, then $\mathscr{I}_Z^{<m_\Lambda>}$ is nothing but the $m$-th symbolic power $\mathscr{I}_Z^{<m>}$ of $\mathscr{I}_Z$.

\begin{remark}
$(1)$ The inequality $\text{ord}_xf\geq m_k$ means that all the partials of $f$ of order $<m_k$ vanish at $x$.

$(2)$ If the inequality $\text{ord}_xf\geq m$ holds at a general point $x\in Z_k$ for every $k$, by upper semi-continuity, it holds at every point of $Z$.
\end{remark}

\begin{remark}
Let $\pi:\widetilde X\to X$ be a log resolution of $\mathscr{I}_Z$ obtained by first blowing up $X$ along $Z$. Since $Z_k$ is generically smooth for every irreducible component $Z_k$ of $Z$, there is a unique irreducible component of the exceptional divisor of $\text{Bl}_Z(X)$ mapping onto $Z_k$, which determines an irreducible divisor $E_k\subset\widetilde X$. Then, by the definition, it follows that
$$\mathscr{I}_Z^{<m_\Lambda>}=\pi_*\mathcal{O}_{\widetilde X}\Big(-\sum_km_kE_k\Big),$$
where the sum does make sense in view of that the family $(Z_k)_{k\in\Lambda}$ is locally finite.
\end{remark}

\begin{definition}
Let $X$ be a complex manifold, $\mathscr{I}\subset\mathcal{O}_X$ an ideal sheaf and $Z\subset X$ an irreducible analytic set. The \emph{order of vanishing} $\text{ord}_Z(\mathscr{I})$ of $\mathscr{I}$ along $Z$ is the largest integer $m$ such that $\mathscr{I}\subset\mathscr{I}_Z^{<m>}$. In addition, we set $$\text{ord}_Z(c\cdot\mathscr{I}):=c\cdot\text{ord}_Z(\mathscr{I})$$ for any real number $c>0$.
\end{definition}

\begin{definition} (cf. \cite{De10}).
Let $X$ be a complex manifold, $Z\subset X$ an irreducible analytic set. Let $\varphi$ be a (quasi-) plurisubharmonic function on $X$ and $x\in X$ a point.

The \emph{Lelong number} $\nu_x(\varphi)$ of $\varphi$ at $x$ is defined to be
$$\nu_x(\varphi):=\sup\{\gamma\ge0\,|\,\varphi(z)\leq\gamma\log|z-x|+O(1)\ \mbox{near}\ x \},$$
on some coordinate neighborhood of $x$. We set $\nu_{x}(\varphi)=+\infty$ if $\varphi\equiv-\infty$.

 The \emph{generic Lelong number} of $\varphi$ along $Z$ is defined as $$\nu_Z(\varphi):=\inf\limits_{z\in Z}\nu_z(\varphi).$$
\end{definition}

\begin{remark} \label{genericLelong}
$(1)$ By the upper semi-continuity of Lelong number, it follows that $\nu_Z(\varphi)=\nu_z(\varphi)$ for a general point $z\in Z$; specifically, for any point outside a countable union of proper analytic subsets of $Z$. In addition, for any $x\in X$, we have
$$\frac{1}{n}\nu_x(\varphi)\leq c_x^{-1}(\varphi)\leq\nu_x(\varphi).$$

$(2)$ Let $z_0\in Z$ be a regular point. Then, we can deduce that
$$\nu_Z(\varphi)=\max\{\gamma\ge0\,|\,\varphi(z)\leq\gamma\log|\mathscr{I}_Z|+O(1)\ \mbox{near}\ z_0 \}.$$
\end{remark}

\begin{definition}
Let $X$ be a complex space (not necessarily reduced). A \emph{graded system of ideal sheaves} $\frak{a}_{\bullet}=\{\frak{a}_m\}$ on $X$ is a collection of coherent ideal sheaves $\frak{a}_m\subset\mathcal{O}_X\ (m\in\mathbb{N})$ such that
$$\frak{a}_0=\mathcal{O}_X\ \text{and}\ \frak{a}_i\cdot\frak{a}_j\subset\frak{a}_{i+j},\ \forall i,j\geq1.$$
\end{definition}

\begin{definition}
The \emph{Rees algebra $\emph{Rees}(\frak{a}_{\bullet})$ of $\frak{a}_{\bullet}$} is the graded $\mathcal{O}_X$-algebra
$$\text{Rees}(\frak{a}_{\bullet}):=\bigoplus\limits_{m\geq0}\frak{a}_m.$$

A graded system of ideal sheaves $\frak{a}_{\bullet}$ is \emph{finitely generated} on $X$ if $\text{Rees}(\frak{a}_{\bullet})$ is finitely generated as a graded $\mathcal{O}_X$-algebra in the sense that there is an integer $m_0$ such that $\text{Rees}(\frak{a}_{\bullet})$ is generated as an $\mathcal{O}_X$-algebra by its terms of degrees $\leq m_0$.
\end{definition}

Let $X$ be a complex manifold and $\frak{a}_{\bullet}=\{\frak{a}_m\}$ a graded system of ideal sheaves on $X$. We define a \emph{Siu plurisubharmonic function}
$\varphi_{\frak{a}_{\bullet}}$ associated to $\frak{a}_{\bullet}$ by (cf. \cite{KimSeo19})
\begin{equation*}
\begin{split}
\varphi_{\frak{a}_{\bullet}}&:=\log\left(\sum\limits_{m=1}^{\infty}\varepsilon_m|\frak{a}_{m}|^{\frac{1}{m}}\right)\\
&\ =\log\left(\sum\limits_{m=1}^{\infty}\varepsilon_m(|f_{m1}|+\cdots+|f_{mr_m}|)^{\frac{1}{m}}\right)
\end{split}
\end{equation*}
on a domain $\Omega\subset X$, where every term $\frak{a}_m$ is an ideal sheaf with a choice of finite generators $f_{m1},...,f_{mr_m}$, and $\varepsilon_m$ approach $0$ so fast as $m\to\infty$ that the infinite series locally converges uniformly.

\begin{lemma} \label{Siu_finite}
Let $\frak{a}_{\bullet}=\{\frak{a}_m\}$ be a finitely generated graded system of ideal sheaves on $\Omega\ni o$ in $\mathbb{C}^n$. Then, there exists $m_0>0$ and one Siu plurisubharmonic function $\varphi_{\frak{a}_{\bullet}}$ associated to $\frak{a}_{\bullet}$ such that $\varphi_{\frak{a}_{\bullet}}=\varphi_{\frak{a}_{\bullet},m_0}+O(1)$ near $o$ with $\varphi_{\frak{a}_{\bullet},m_0}:=\log(\sum\limits_{m=1}^{m_0}\varepsilon_m|\frak{a}_{m}|^{\frac{1}{m}})$.
\end{lemma}

\begin{proof}
Since $\frak{a}_{\bullet}$ is finitely generated on $\Omega$, there exists $m_0>0$ such that for any $m>m_0$,
$$\frak{a}_m=\sum\limits_{k_1+2k_2+\cdots+m_0k_{m_0}=m}\frak{a}_1^{k_1}\cdot\frak{a}_2^{k_2}\cdots\frak{a}_{m_0}^{k_{m_0}}.$$
Note that $\mathfrak{a}_m^{(m-1)!}\subset\mathfrak{a}_{m!}$ by the definition of graded system of ideal sheaves, and then we can conclude that for some constant $C>0$ (may depend on $m$),
$$|\frak{a}_m|^{(m-1)!}\leq C\cdot\left(|\frak{a}_1|+|\frak{a}_2|^{\frac{1}{2}}+\cdots+|\frak{a}_{m_0}|^{\frac{1}{m_0}}\right)^{m!},$$
which implies that
$$|\frak{a}_m|^{\frac{1}{m}}\leq C^{\frac{1}{m!}}\cdot\left(|\frak{a}_1|+|\frak{a}_2|^{\frac{1}{2}}+\cdots+|\frak{a}_{m_0}|^{\frac{1}{m_0}} \right)$$
for all $m>m_0$ (shrinking $\Omega$ if necessary). Thus, from the definition of Siu plurisubharmonic function, after shrinking $\varepsilon_m$, it follows that
$\varphi_{\frak{a}_{\bullet}}=\varphi_{\frak{a}_{\bullet},m_0}+O(1)$ on $\Omega$.
\end{proof}

\begin{lemma} \emph{(cf. \cite{KimSeo19}, Theorem 2.2)} \label{equivalence}
For each Siu plurisubharmonic function $\varphi_{\frak{a}_{\bullet}}$ defined above on $\Omega\subset X$, it follow that $\mathscr{I}(c\cdot\varphi_{\frak{a}_{\bullet}})=\mathscr{I}(c\cdot\frak{a}_{\bullet})$ on $\Omega$ for all $c>0$.
\end{lemma}

At the end of this part, we prove the following result on the finite generation of certain graded system of ideal sheaves.

\begin{proposition} \emph{(Chenyang Xu).} \label{Xu}
Let $X$ be a smooth complex quasi-projective variety and $A\subset X$ an algebraic subvariety (\emph{not} necessarily irreducible). Let $\frak{a}_{\bullet}=\{\frak{a}_m\}$ be a graded sequence of ideals on $X$ given by $\frak{a}_m=\mathscr{I}_A^{<m_\Lambda>}$ with $m_k=e_km$ for each $k$ and $e_k$ the codimension of $A_k$ in $\Omega$.
If $\emph{LCT}(\frak{a}_{\bullet})=1$, then $\frak{a}_{\bullet}$ is finitely generated on $X$.
\end{proposition}

\begin{proof}
Let $A_k\ (1\leq k\leq s)$ be all the irreducible components of $A$, and $E_k$ the divisor obtained by blowing up the generic point of $A_k$. Write $E=\sum_{1\leq k\leq s}E_k$ and $\frak{a}_m=\mu_*\mathcal{O}_{X'}\big(-m\sum_{1\leq k\leq s}e_kE_k\big)$, where $\mu:X'\to X$ is a model which contains all $E_k$.

For each $E_k$, we know $\text{ord}_{E_k}(\frak{a}_m)\geq e_km$ by the definition, thus $$\text{ord}_{E_k}(\frak{a}_{\bullet})=\lim\limits_{m\to\infty}\frac{1}{m}\text{ord}_{E_k}(\frak{a}_m)\geq e_k.$$
We also have the log discrepancy $A_X(E_k)=e_k$. Therefore, the assumption implies that for each $k$,
$$1=\text{LCT}(\frak{a}_{\bullet})\leq\frac{A_X(E_k)}{\text{ord}_{E_k}(\frak{a}_{\bullet})}\leq\frac{e_k}{e_k}=1,$$
which implies that $\text{ord}_{E_k}(\frak{a}_{\bullet})=e_k$ and $E_k$ computes the log canonical threshold of $\frak{a}_{\bullet}$.

The rest is a standard argument using the minimal model program: for any $0<\varepsilon<1$, there exists sufficiently large $m_0$ such that $$\text{LCT}(\frac{1}{m}\frak{a}_m)>1-\varepsilon,\ \forall m\geq m_0.$$ Thus we know that for each $E_k$, the log discrepancy
\begin{equation*}
\begin{split}
0<A_{X,\frac{1-\varepsilon}{m}\frak{a}_m}(E_k)&:=A_X(E_k)-\frac{1-\varepsilon}{m}\text{ord}_{E_k}(\frak{a}_m)\\
&\ \leq e_k-(1-\varepsilon)e_k\\
&\ =\varepsilon e_k,
\end{split}
\end{equation*}
where the first inequality uses that $\text{LCT}(\frac{1}{m}\frak{a}_m)>1-\varepsilon$, and the middle inequality follows from that \begin{equation*}
\begin{split}
\frac{1}{m}\text{ord}_{E_k}(\frak{a}_m)&\geq\inf\limits_{m\geq1}\frac{\text{ord}_{E_k}(\frak{a}_m)}{m}\\
&=\text{ord}_{E_k}(\frak{a}_{\bullet})\\
&=e_k.
\end{split}
\end{equation*}

We can choose $\varepsilon$ sufficiently small such that $\varepsilon e_k<1$ for each $k$, then by Corollary 1.4.3 in \cite{BCHM}, we know that there exists a model $\pi:\widetilde X\to X$ such that the exceptional divisor $\text{Ex}(\pi)$ is precisely $E$, the sum of $E_k$, and $-E$ is $\pi$-ample, which implies that $\frak{a}_{\bullet}$ is finitely generated on $X$.
\end{proof}

\begin{remark}
As we will see, by combining (2) of Theorem \ref{SingPairs} with (1) of Remark \ref{analyticity}, the algebraic variety $A$ in $X$ with the reduced complex structure is a weakly normal complex space.
\end{remark}

\subsection{The relation between log canonical locus and symbolic powers}

This part is devoted to the algebraic properties of log canonical locus of plurisubharmonic functions. At first, by combining Demailly's analytic approximation of plurisubharmonic functions with properties of asymptotic multiplier ideals, we are able to establish the following result, which is important for us to study the local structure of log canonical singularities for general plurisubharmonic functions.

\begin{theorem} \label{SiuPsh}
Let $o\in\Omega\subset\mathbb{C}^{n}$ be the origin, $u\in\text{\emph{Psh}}(\Omega)$ with $c_o(u)=1$ and $A=N(\mathscr{I}(u))$. Then, it follows that
$$c_o(\varphi_{\frak{a}_{\bullet}})=\emph{LCT}_o(\frak{a}_{\bullet})=1\ \text{and}\ \mathscr{I}(\varphi_{\frak{a}_{\bullet}})_o=\mathscr{I}(\frak{a}_{\bullet})_o=\mathscr{I}(u)_o,$$
where $\frak{a}_{\bullet}=\{\frak{a}_m\}$ is the graded system of ideal sheaves on $\Omega$ given by $\frak{a}_m=\mathscr{I}_A^{<m_\Lambda>}$ with $m_k=e_km$ for each $k$ and $e_k$ the codimension of $A_k$ in $\Omega$.
\end{theorem}

\begin{remark}
As an enhanced version of Theorem \ref{SiuPsh}, we expect that there exists a plurisubharmonic function $u_A$ with analytic singularities near $o$ such that $c_o(u_A)=c_o(u)=1$ and $\mathscr{I}(u_A)_o=\mathscr{I}(u)_o$, which enables us to give a positive solution to Question \ref{Q_WNormal} by (2) of Theorem \ref{SingPairs}.
\end{remark}

\begin{corollary} \label{LCI}
With the same hypothesis as in Theorem \ref{SiuPsh}. Then:

$(1)$ The following statements are equivalent:\\

$(a)$ There exists a plurisubharmonic function $u_A$ with analytic singularities near $o$, such that $c_o(u_A)=1$ and $\mathscr{I}(u_A)_o=\mathscr{I}(u)_o$.

$(b)$ There is $m_0>0$ such that $c_o(\frac{1}{m_0}\log|\frak{a}_{m_0}|)=1$.\\

$(2)$ If $A$ is locally a (scheme-theoretic) complete intersection, then
$$c_o(\log|\frak{a}_1|)=1\ \text{and}\ \mathscr{I}(\log|\frak{a}_1|)_o=\mathscr{I}(u)_o.$$
\end{corollary}

\begin{remark} \label{analyticity}
$(1)$ Suppose that $\frak{a}_\bullet$ is finitely generated (e.g., $A$ is locally a complete intersection). Then, by Lemma \ref{Siu_finite} there exists $m_0>0$ such that the Siu plurisubharmonic function $\varphi_{\frak{a}_{\bullet}}=\varphi_{\frak{a}_{\bullet},m_0}+O(1)$ (shrinking $\varepsilon_m$ if necessary), which implies that
$$c_o(u_A)=1\ \text{and}\ \mathscr{I}(u_A)_o=\mathscr{I}(u)_o$$
by taking $u_A=\varphi_{\frak{a}_{\bullet},m_0}$.

$(2)$ If $(A,o)$ is of dimension $n-2$ and embedding dimension $n-1$, by the above Corollary, we obtain that
$$c_o(\log|\frak{a}_1|)=1\ \text{and}\ \mathscr{I}(\log|\frak{a}_1|)_o=\mathscr{I}(u)_o.$$
In particular, if $n=3$ and $(A,o)$ is a singular curve of embedding dimension $2$, then $(A,o)$ is weakly normal by (2) of Theorem \ref{SingPairs}.
\end{remark}

\begin{corollary} \label{isolated}
Let $u\in\text{\emph{Psh}}(\Omega)$ such that $c_o(u)=1$ and $A=N(\mathscr{I}(u))$. Then, we can deduce that

$(1)$ If $\dim_oA=n-1$, then any union of $(n-1)$-dimensional irreducible components of $(A,o)$ is a semi-log-canonical hypersurface.

$(2)$ Assume that $o$ is an isolated singularity of $A$. Then the multiplier ideal subscheme $V(u)$ is weakly normal near $o$.
\end{corollary}

\begin{remark}
As an immediate consequence of Corollary \ref{isolated}, we can obtain Theorem 1.2 in \cite{G-L_LCT16}.
\end{remark}

\subsection{A solution to Question \ref{Q_WNormal} in dimension three}

In the last part, we give a positive answer to Question \ref{Q_WNormal} in the three dimensional case.

\subsubsection{Proof of Theorem \ref{3D}}

By Theorem \ref{SiuPsh}, we have
$$c_o(\varphi_{\frak{a}_{\bullet}})=1\ \text{and}\ \mathscr{I}(\varphi_{\frak{a}_{\bullet}})_o=\mathscr{I}(u)_o,$$
where $\frak{a}_{\bullet}=\{\frak{a}_m\}$ is the graded system of ideal sheaves on $\Omega$ given by $\frak{a}_m=\mathscr{I}_A^{<m_\Lambda>}$ with $m_k=e_km$ for each $k$ and $e_k$ the codimension of $A_k$ in $\Omega$. Thanks to Lemma \ref{Siu_finite} and (2) of Theorem \ref{SingPairs}, it is sufficiently to show that $\frak{a}_{\bullet}$ is finitely generated on some neighborhood of $o$.\\

(i) When $\dim_o A=1$, it follows that $o$ is at most an isolated singularity of $A$. Then, we can achieve the finite generation by a combination of Artin's algebraization theorem with Proposition \ref{Xu}.

(ii) When $(A,o)$ is of pure dimension 2, the desired result follow from the fact that $$c_o(\mathscr{I}(u))\geq c_o(u)=1;$$
and in further $(A,o)$ is algebraic (cf. \cite{K-SB, L-R}).

(iii) When $\dim_oA=2$ and $(A,o)$ is not of pure dimension. After shrinking $\Omega$, we may assume that each irreducible component of $A$ contains the origin $o$. Let $A=A_1\cup A_2$ with $\dim_oA_k=3-k\ (k=1,2)$, and then we achieve the finite generation of $\frak{a}_\bullet$ near $o$ by Proposition \ref{Xu} and the algebraicity of $(A_k,o)$.
\hfill$\Box$

\begin{remark}
One can refer to Remark \ref{alter} in Appendix \ref{A} for an alternative argument of the proof of Theorem \ref{3D}.
\end{remark}

\subsubsection{Proof of Corollary \ref{union}}

Put $Z=\bigcup\limits_{\alpha=1}^tA_{k_\alpha}$, where $A_{k_\alpha}$ is an irreducible component of $A$ through $o$ for each $\alpha$ (shrinking $\Omega$ if necessary) and of codimension $e_{k_\alpha}$. Let $m_\Lambda=(e_{k_1},\dots,e_{k_t})$ and consider the plurisubharmonic function $$\widetilde u:=\log\left(|\mathscr{I}_Z^{<m_\Lambda>}|+e^{u}\right),$$
where $\mathscr{I}_Z^{<m_\Lambda>}$ is the $m_\Lambda$-th symbolic power of $\mathscr{I}_Z$. Thanks to Theorem \ref{3D}, it is sufficiently to prove that $c_o(\widetilde u)=1$ and $\mathscr{I}(\widetilde u)_o=\mathscr{I}_{Z,o}$, which is a consequence of Demailly's analytic approximation of plurisubharmonic functions and $(2)$ of Remark \ref{genericLelong}
\hfill$\Box$

\subsubsection{Proof of Theorem \ref{char_3D}}

$\textbf{(\emph{a})}$ When $(A,o)$ is of pure dimension 2, the desired result immediately follows from (ii) in the proof of Theorem \ref{3D}.

$\textbf{(\emph{b})}$ When $\dim_oA=2$ and $(A,o)$ is not of pure dimension. Let $$(A,o)=(A_1,o)\cup(A_2,o)$$ with $\dim_o(A_k,o)=3-k\ (k=1,2)$. Thanks to Theorem \ref{3D}, we can conclude that $V(u)$ is weakly normal near $o$.

Therefore, in view of Lemma 2.3 in \cite{AL}, it follows from the fact that $(A_1,o)\cap(A_2,o)=\{o\}$ that $T_oA_1\cap T_oA_2=0$, where $T_oA_k\ (k=1,2)$ is the Zariski tangent space of $A_k$ at $o$, which implies that both $(A_1,o)$ and $(A_2,o)$ are regular, and they intersect transversely at $o$. Thus, up to a change of variables at $o$, we obtain that $\mathscr{I}(u)_o=(z_1z_2,z_1z_3)\cdot\mathcal{O}_{3}$.

$\textbf{(\emph{c})}$ When $\dim_oA=1$, i.e., $(A,o)$ is a curve with embedding dimension at most 3. Then, Theorem \ref{3D} implies that $V(u)$ is weakly normal at $o$. Thus, we obtain that $(A,o)=(N(z_1,z_2),o)$, $(N(z_1,z_2z_3),o)$ or $(N(z_1z_2,z_1z_3,z_2z_3),o)$ in some local coordinates near $o$, which implies that $\mathscr{I}(u)_o=(z_1,z_2)\cdot\mathcal{O}_{3}$, $(z_1,z_2z_3)\cdot\mathcal{O}_{3}$ or $(z_1z_2,z_1z_3,z_2z_3)\cdot\mathcal{O}_{3}$, respectively.

$\textbf{(\emph{d})}$ When $\dim_oA=0$, the desired result is straightforward.
\hfill$\Box$

\begin{remark}
We refer to \cite{G-L_LCT19} for an alternative argument of the case $\textbf{(\emph{b})}$ by the so-called ``inversion of adjunction'' (\cite{DK01}, Theorem 2.5).
\end{remark}

\subsubsection{Proof of Corollary \ref{LCS_pairs3D}}

Since the statement is local, we may assume that $V(\frak{a}_\bullet)$ is the multiplier ideal subscheme associated to $\frak{a}_\bullet$ on a bounded Stein domain $\Omega\subset\mathbb{C}^3$. Let $\varphi_{\frak{a}_\bullet}$ be a Siu plurisubharmonic function associated to $\frak{a}_\bullet$ on $\Omega$. Then, we can deduce from Lemma \ref{equivalence} that $\mathscr{I}(\varphi_{\frak{a}_\bullet})=\mathscr{I}(\frak{a}_\bullet)$ on $\Omega$ and $c_x(\varphi_{\frak{a}_\bullet})=\text{LCT}(\frak{a}_\bullet)=1$ for any $x\in V(\frak{a}_\bullet)$. Thus, we can immediately achieve the desired result by Theorem \ref{3D} and Theorem \ref{char_3D}.

\hfill$\Box$

\begin{remark}
Similar to Corollary \ref{union}, we obtain that any union of irreducible components of $V(\frak{a}_\bullet)$ with the reduced complex structure is weakly normal as well.
\end{remark}

\appendix
   \renewcommand{\appendixname}{Appendix~\Alph{section}}

\section{Constructibility of log canonical threshold of graded system of ideal sheaves from symbolic powers} \label{A}

In this appendix, we are ready to discuss something on the constructibility of log canonical thresholds from symbolic powers, which turns out to be helpful for us to study Question \ref{Q_WNormal}.

\begin{definition} (cf. \cite{Loja91}).
Let $X$ be a complex manifold. A subset $Z\subset X$ is \emph{analytically constructible} if $Z=\bigcup\limits_{\lambda\in\Lambda}(V_\lambda\backslash W_\lambda)$, where $\{V_\lambda\}_{\lambda\in\Lambda}$ is a locally finite family of irreducible analytic sets and $W_\lambda\subsetneqq V_\lambda$ is an analytic set for every $\lambda\in\Lambda$.

A function $f: X\to[-\infty,+\infty]$ is \emph{analytically constructible} if each set $f^{-1}(t)$ is constructible for any $t\in[-\infty,+\infty]$, and the family $\{f^{-1}(t)\}_{t\in[-\infty,+\infty]}$ is locally finite on $X$.

Similarly, we can define a counterpart in the algebraic setting. In particular, every algebraically constructible set or function is analytically constructible.
\end{definition}

\begin{remark}
$(1)$ The difference, finite intersection and union of analytically constructible sets is also an analytically constructible set.

$(2)$ Let $\Omega\subset\subset X$ be a relatively compact domain. Then, the range $f(\Omega)$ of analytically constructible function $f|_\Omega$ is a finite subset in $[-\infty,+\infty]$.
\end{remark}

\begin{remark} (Chevalley-Remmert; cf. \cite{Loja91}, p. 291). \label{CheRe}
Let $\pi:\widetilde X\to X$ be a proper holomorphic mapping of complex manifolds. Then the image of every analytically constructible set in $\widetilde X$ is an analytically constructible set in $X$.
\end{remark}

\begin{proposition} \label{constructibility_3D}
Let $X$ be a complex manifold of dimension three and $Z\subset X$ an analytic set. Then, the log canonical threshold $\emph{LCT}_x(\frak{a}_{\bullet})$ is an analytically constructible function on $X$, where $\frak{a}_{\bullet}=\{\frak{a}_m\}$ is the graded sequence of ideals on $X$ given by $\frak{a}_m=\mathscr{I}_Z^{<m_\Lambda>}$ with $m_k=e_km$ for each $k$ and $e_k$ the codimension of $Z_k$ in $X$.
\end{proposition}

\begin{proof}
We can achieve this by a similar argument to the proof of Theorem \ref{3D}. Without loss of generality, we may assume that $\dim Z_k>0$ for each irreducible component $Z_k$ of $Z$.\\

(i) When $\dim Z=1$, i.e., $Z\subset X$ is a curve. Then, the singular locus of $Z$ is at most a discrete subset in $X$. Note that $\text{LCT}_x(\frak{a}_{\bullet})=+\infty$ iff $x\in X\backslash Z$, and $\text{LCT}_x(\frak{a}_{\bullet})=1$ for any point $x\in Z_{\text{reg}}$. Hence, we can deduct that $\text{LCT}_x(\frak{a}_{\bullet})$ is an analytically constructible function on $X$.\\

(ii) When $\dim Z_k=2$ for each $k$, i.e., $Z$ is a hypersurface of $X$. Thus, we obtain that the symbolic power of $\mathscr{I}_Z$ coincides with the ordinary power, i.e., $\mathscr{I}_Z^{<m>}=\mathscr{I}_Z^{m}$ for all $m$, which implies that $\text{LCT}_x(\frak{a}_{\bullet})=\text{LCT}_x(\mathscr{I}_Z)$ for any point $x\in X$. Let $\pi:\widetilde X\to X$ be a log resolution of $\mathscr{I}_Z$, and then by a standard argument on the computation of $\text{LCT}_x(\mathscr{I}_Z)$ via the log resolution $\pi$ (see \cite{DK01}, Proposition 1.7), we conclude the constructibility of $\text{LCT}_x(\frak{a}_{\bullet})$.\\

(iii) When $\dim Z=2$ and $\dim Z_k=1$ for some $k$. Let $\widetilde Z_\alpha\ (\alpha=1,2)$ be the union of all $\alpha$-codimensional irreducible components of $Z$, and consider the graded system of ideal sheaves $\frak{a}_\bullet',\frak{a}_\bullet''$ on $X$ given by ${\frak{a}}_m'=\mathscr{I}_{\widetilde Z_1}^{<m>}$ and ${\frak{a}}_m''=\mathscr{I}_{\widetilde Z_2}^{<2m>}$ respectively. Thus, it follows from (i) and (ii) that $\text{LCT}_x(\frak{a}_{\bullet}')$ and $\text{LCT}_x(\frak{a}_{\bullet}'')$ are analytically constructible functions on $X$.

On the other hand, we note that $\text{LCT}_x(\frak{a}_{\bullet})=\text{LCT}_x(\frak{a}_{\bullet}')$ for any $x\in\widetilde Z_1\backslash\widetilde Z_2$ and $\text{LCT}_x(\frak{a}_{\bullet})=\text{LCT}_x(\frak{a}_{\bullet}'')$ for any $x\in\widetilde Z_2\backslash\widetilde Z_1$. Finally, combining the discreteness of analytic set $\widetilde Z_1\cap \widetilde Z_2$ in $X$ and the constructibility of $\text{LCT}_x(\frak{a}_{\bullet}')$ and $\text{LCT}_x(\frak{a}_{\bullet}'')$, we conclude that the function $\text{LCT}_x(\frak{a}_{\bullet})$ is constructible on $X$.
\end{proof}

\begin{remark} \emph{(An alternative argument on the proof of Theorem \ref{3D}).} \label{alter}
At first, we observe that $o$ is at most an algebraic singularity of $A$. Then, there exists an affine algebraic variety $\widehat A\subset\mathbb{C}^n$ such that $(\widehat A,o)=(A,o)$ and we may assume that all irreducible components of $\widehat A$ contain the origin $o$. Thus, by Theorem \ref{SiuPsh}, we have $$\text{LCT}_o(\widehat{\frak{a}}_\bullet)=\text{LCT}_o({\frak{a}}_\bullet)=1\ \text{and}\ \mathscr{I}(\varphi_{\widehat{\frak{a}}_{\bullet}})_o=\mathscr{I}(u)_o,$$
where $\widehat{\frak{a}}_{\bullet}=\{\widehat{\frak{a}}_m\}$ is the graded system of ideal sheaves on $\mathbb{C}^n$ given by $\widehat{\frak{a}}_m=\mathscr{I}_{\widehat A}^{<m_\Lambda>}$.

Combining the lower semi-continuity of log canonical thresholds with Proposition \ref{constructibility_3D}, we are able to infer that $\text{LCT}(\widehat{\frak{a}}_\bullet)=1$ on an algebraic Zariski open neighborhood $U$ of $o$ in $\mathbb{C}^n$. Thus, by Proposition \ref{Xu}, it follows that $\widehat{\frak{a}}_{\bullet}$ is finitely generated on $U$. Therefore, $\frak{a}_{\bullet}$ is finitely generated on a neighborhood of $o$ (in the complex topology), which implies that $c_o(\varphi_{\frak{a}_{\bullet},m_0})=1$ for some $m_0>0$, and $\mathscr{I}(\varphi_{\frak{a}_{\bullet},m_0})=\mathscr{I}(u)$ near $o$ by Lemma \ref{Siu_finite}. Finally, using (2) of Theorem \ref{SingPairs}, we conclude that the multiplier ideal subscheme $V(u)$ is weakly normal, shrinking $\Omega$ if necessary.
\hfill$\Box$
\end{remark}

In view of Proposition \ref{constructibility_3D}, it is natural to raise the following question:

\begin{question}  \label{Q_constructibility}
Let $X$ be a complex manifold and $Z\subset X$ an analytic set. Let $\frak{a}_{\bullet}=\{\frak{a}_m\}$ be the graded sequence of ideal sheaves on $X$ given by $\frak{a}_m=\mathscr{I}_Z^{<m_\Lambda>}$, where $m_k=e_km$ for each $k$ and $e_k$ is the codimension of $Z_k$ in $X$. Is the log canonical threshold $\emph{LCT}_x(\frak{a}_{\bullet})$ an analytically constructible function on $X$?
\end{question}

\begin{remark}
$(1)$ From the argument in the proof of Proposition \ref{constructibility_3D}, we can deduce that the answer of Question \ref{Q_constructibility} is positive if $Z\subset X$ is locally a complete intersection or all singularities of $Z$ are isolated.

$(2)$ Analogous to Remark \ref{alter}, we are able to derive that if $(A,o)$ is algebraic, then a positive answer of Question \ref{Q_constructibility} implies a positive solution to Question \ref{Q_WNormal}.
\end{remark}

Finally, by combining the argument in Remark \ref{alter} with Lemma \ref{equivalence}, we can deduce the following observation on log canonical locus of pair associated to graded system of ideals in the algebraic setting.

\begin{proposition} \label{LCS_pairs}
Let $X$ be a smooth complex projective variety and $\frak{a}_{\bullet}=\{\frak{a}_m\}$ any graded system of ideals on $X$ such that $\emph{LCT}(\frak{a}_{\bullet})=1$. If the answer of Question \ref{Q_constructibility} is positive, then any union of the irreducible components of multiplier ideal subscheme $V(\frak{a}_{\bullet})$ is weakly normal.
\end{proposition}
\newpage

\section{Semi-log-canonical hypersurface singularities of dimension two}  \label{B}

\begin{table}[htbp]
\[
  \renewcommand\arraystretch{1.5}
  \newcolumntype{V}{!{\vrule width 1.7pt}}
  \begin{array}{VcVcVlV}
    \Xhline{1.7pt}
      \mathfrak{name}                          & \mathfrak{symbol}        & \multicolumn{1}{cV}{\mathfrak{defining\ function}\ h\in\mathbb{C}[z_1,z_2,z_3]}\\
    \Xhline{1.7pt}
      \text{smooth}                            & A_0                      & z_1\\
    \Xhline{1pt}
      \multirowcell{5}{\text{Du Val}}          & A_n                      & z_1^2+z_2^2+z_3^{n+1},\qquad\qquad\,\ n\geq1 \\\cline{2-3}
                                               & D_n                      & z_1^2+z_2^2z_3+z_3^{n-1},\qquad\quad\,\ n\geq4\\\cline{2-3}
                                               & E_6                      & z_1^2+z_2^3+z_3^4\\\cline{2-3}
                                               & E_7                      & z_1^2+z_2^3+z_2z_3^3\\\cline{2-3}
                                               & E_8                      & z_1^2+z_2^3+z_3^5\\\cline{2-3}
    \Xhline{1pt}
      \multirowcell{3}{\text{simple elliptic}} & X_{1,0}                  & z_1^2+z_2^4+z_3^4+\lambda z_1z_2z_3,\quad\lambda^4\not=64 \\ \cline{2-3}
                                               & J_{2,0}                  & z_1^2+z_2^3+z_3^6+\lambda z_1z_2z_3,\quad\lambda^6\not=432 \\ \cline{2-3}
                                               & T_{3,3}                  & z_1^3+z_2^3+z_3^3+\lambda z_1z_2z_3,\quad\lambda^3+27\not=0 \\ \cline{2-3}
    \hline
      \text{cusp}                              & T_{p,q,r}                & z_1z_2z_3+z_1^p+z_2^q+z_3^r,\ \,\quad\frac{1}{p}+\frac{1}{q}+\frac{1}{r}<1\\
    \Xhline{1pt}
      \text{normal crossing}                   & A_{\infty}               & z_1^2+z_2^2\\
    \hline
      \text{pinch point}                       & D_{\infty}               & z_1^2+z_2^2z_3\\
    \hline
      \multirowcell{5}{\text{degenerate cusp}} & T_{2,\infty,\infty}      & z_1^2+z_2^2z_3^2 \\\cline{2-3}
                                               & T_{2,q,\infty}           & z_1^2+z_2^q+z_2^2z_3^2,\qquad\qquad\ \ q\geq3\\\cline{2-3}
                                               & T_{\infty,\infty,\infty} & z_1z_2z_3\\\cline{2-3}
                                               & T_{p,\infty,\infty}      & z_1z_2z_3+z_1^p,\qquad\qquad\qquad p\geq3\\\cline{2-3}
                                               & T_{p,q,\infty}           & z_1z_2z_3+z_1^p+z_2^q,\qquad\quad\,\ \ q\geq p\geq3\\\cline{2-3}
    \Xhline{1.7pt}
  \end{array}
\]
\caption{\emph{Semi-log-canonical hypersurface singularities of dimension two}}
\label{list}
\end{table}

\vspace{.1in} {\em Acknowledgements}.
The author would like to sincerely thank Prof. Xiangyu Zhou and Qi'an Guan for their generous support and encouragements. He is very grateful to Prof. Chenyang Xu for valuable discussions and suggestions, especially for providing a proof of Proposition \ref{Xu}. He is also indebted to Prof. Jean-Pierre Demailly for helpful correspondence, and Prof. William Allen Adkins for sharing his works.

\end{document}